\documentclass[12pt]{amsart}
\usepackage[english]{babel}
\usepackage{paralist,url,verbatim, anysize}
\usepackage{amscd}
\usepackage[all,cmtip]{xy} 
\usepackage{amsmath,amsthm,amssymb,amsfonts,amsaddr}
\usepackage[usenames, dvipsnames]{color}
\usepackage{graphicx}
\usepackage{hyperref}
\usepackage{euler, latexsym, epsfig,epic}

\makeatletter
\let\@fnsymbol\@arabic
\makeatother

\theoremstyle{plain}
\newtheorem{theorem}{Theorem}[section]
\newtheorem{proposition}[theorem]{Proposition}
\newtheorem{corollary}[theorem]{Corollary}
\newtheorem{lemma}[theorem]{Lemma}
\newtheorem{remark}[theorem]{Remark}
\newtheorem{assumption}[theorem]{Assumption}

 \def\nuu{u}

\def\m{\mathfrak{m}}

\def\0{\boldsymbol{0}}

\def\K{\mathcal{K}}

\def\V{\mathcal{V}}

\DeclareMathOperator*{\Ass}{Ass}
\DeclareMathOperator*{\reg}{reg}

\DeclareMathOperator*{\LCM}{LCM}

\def\NN{\mathbb{N}}
\def\ZZ{\mathbb{Z}}
\DeclareMathOperator{\Sym}{Sym}
\DeclareMathOperator{\rk}{rk}
\DeclareMathOperator{\HS}{HS}
\DeclareMathOperator{\PM}{P}
\DeclareMathOperator{\pd}{projdim}
\long\def\answer#1{}

\long\def\comment#1{}

\usepackage[margin=1in]{geometry}

\title{Resolution of ideals associated to  subspace arrangements}

\author{Aldo Conca}
\address{Dipartimento di Matematica,  Universit\`a di Genova, Via Dodecaneso 35, 16146 Genova, Italy}
\email{conca@dima.unige.it}
\author{Manolis C. Tsakiris}
\address{School of Information Science and Technology, ShanghaiTech University, No.393 Huaxia Middle Road, Pudong Area, Shanghai, China}
\address{Dipartimento di Matematica,  Universit\`a di Genova, Via Dodecaneso 35, 16146 Genova, Italy}
\email{mtsakiris@shanghaitech.edu.cn}

\begin{document}

\begin{abstract}
Let $I_1,\dots,I_n$ be ideals generated by linear forms in a polynomial ring over an infinite field and let $J = I_1 \cdots I_n$.
We describe a  minimal free resolution of $J$ and show that it is supported on a polymatroid  obtained from the underlying representable polymatroid by means of the so-called Dilworth truncation.   
Formulas for the projective dimension and Betti numbers are given in terms of the polymatroid  as well as a characterization of the associated primes. Along the way we show that $J$ has linear quotients. In fact, we do this for a large class of ideals $J_P$, where $P$ is a certain poset ideal associated to the underlying subspace arrangement. 
\end{abstract}

\maketitle

\section*{Introduction} 
A subspace arrangement $\V$ is a finite collection $V_1,\dots, V_n$ of vector subspaces of a given vector space $V$ over a field $K$.  Several  geometric objects can be associated to $\V$  and their investigation has   attracted the attention of many  researchers, see for example  Bj\"orner \cite{B}, De Concini and Procesi \cite{DP} and Bj\"orner, Peeva and Sidman \cite{BPS}.   Subspace arrangements interplay as well with multigraded commutative algebra and geometric computer vision, see \cite{AST, C, CS, CDG,Li}, where a subspace arrangement $\V$ gives rise to  a multigraded $K$-algebra, called the multiview algebra. Another application of subspace arrangements is in the machine learning problem of generalized principal component analysis \cite{VMS} also known as subspace clustering \cite{V}, \cite{TV}. 

In this paper we consider the product $J$ of the ideals $I_i$ generated by the $V_i$'s  in the polynomial ring $S=\Sym_K(V)$.  In \cite{CH} a primary decomposition of $J$ is presented. It is indeed a ``combinatorial" decomposition since the ideals involved are powers of ideals generated by  sums of the $V_i$'s.  From  the primary decomposition one reads immediately that $J$ is saturated from degree $n$. This is the key ingredient of the proof in \cite{CH} asserting the minimal free resolution of $J$ is linear, i.e. the Castelnuovo-Mumford regularity of $J$ is  exactly $n$.  In \cite{D} Derksen proved that the Hilbert function of $J$ is a  combinatorial invariant, that is, it    just depends  of the rank function:
$$\rk_\V: 2^{[n]}\to \NN,  \qquad    A\subseteq [n],   \qquad   \rk_\V(A)=\dim_K \sum_{i\in A} V_i.$$
As observed by Derksen, since the resolution is  linear, this implies that  the algebraic Betti numbers of $J$ are themselves  combinatorial invariants. 
Attached to the rank function we have a discrete polymatroid
$$\PM(\V)=\{ x\in \NN^n : \sum_{i\in A} x_i \leq \rk_\V(A) \mbox{ for all } A\subseteq [n] \}$$
that plays a role in the sequel. 

The goal of the paper is to describe the minimal free resolution and the associated primes of $J$ and  give an explicit formula for the Betti numbers and the projective dimension.  Indeed we prove that the minimal free resolution of $J$ can be realized as a subcomplex of the tensor product of the  Koszul complexes associated with generic generators of the $V_i$. Such a resolution is  supported on the subpolymatroid 
$$\PM(\V)^*=\left\{ x\in \NN^n : \sum_{i\in A} x_i \leq \rk_\V(A)-1 \mbox{ for all } \emptyset\neq A\subseteq [n] \right\}$$
of $\PM(\V)$  whose rank function $\rk_\V^*$ is obtained by  the so-called Dilworth truncation, i.e. 

$$\rk_\V^*(A)= \min\left\{ \sum_{i=1}^p \rk_\V(A_i)-p : A_1,\dots, A_p \mbox{  is a partition of } A\right\}.$$

  It turns out that the (algebraic) Betti numbers $\beta_i(J)$ of $J$ are given by: 

$$\sum_{i\geq 0} \beta_i(J)z^i=\sum_{i\geq 0} 
\gamma_i(\V) (1+z)^i$$ 
where $\gamma_i(\V)  = \# \{ x\in \PM(\V)^*: |x|=i\}$,  and that the projective dimension of $J$ is given by the formula: 
$$\pd J=\rk_\V^*([n])=\min\left \{  \sum_{i=1}^p  \rk_\V(A_i)-p : A_1,\dots, A_p \mbox{ is a partition of } [n] \right \}.$$

Furthermore the associated primes of $J$  are exactly  the ideals  of the form   $\sum_{i \in A} I_i$  where  $A\subseteq  [n]$ such that and $\rk_\V^*(A)=\rk_\V(A)-1$.

As a corollary, we obtain that 
$$\pd J=\pd J^\nuu  \mbox{ and } \Ass(S/J)=\Ass(S/J^\nuu)$$
 hold for every $\nuu>0$. 
 
The formulas for the Betti numbers and the projective dimension and the description of the associated primes hold over any base field, while the  description of the minimal free resolution depends on the  choice of generic bases  (in a precise sense, see \ref{assum}) of the $V_i$'s whose existence is guaranteed  over  an infinite base field.

Our results apply indeed to an entire family of ideals associated with the subspace arrangement that  makes possible inductive arguments. As a by-product we prove that the ideal $J$ has linear quotients. 
 
We thank Prof.~F.~Ardila, Prof.~A.~Fink, and Prof.~S.~Fujishige for useful discussions concerning polymatroids. 

\section{Notation and basic facts} 
Let $K$ be an infinite   field and $V$ a $K$-vector space of dimension $d$. Let $S$ be the symmetric algebra of $V$, i.e. a polynomial ring over $K$ of dimension $d$. Let $\V=V_1,\dots, V_n$ be a collection of non-zero $K$-subspaces of $V$. Let  $d_i=\dim_K V_i$.   
Such a collection $\V$  is called a subspace arrangement of dimension $(d_1,\dots, d_n)$. For $i\in [n]$ let $\{ f_{ij} : j\in [d_i] \}$ be an ordered $K$-basis of $V_i$. The arrangement of vectors 
$$\left\{ f_{ij} : i\in [n] \mbox{ and }   j\in  [d_i] \right\}$$
is called a collection of bases of $\V$. 
 Here and in the following for  $u \in \mathbb{N} $ we denote by $[u]$ the set $\{1,\dots, u\}$.  As usual for $i \in [n]$ we will denote by $e_i \in \mathbb{N}^n$ the vector with zeros everywhere except a $1$ at position $i$ and for $a \in  \NN^n$ we set  $|a|=a_1+\cdots+a_n$.  

For every $a=(a_1,\dots, a_n)  \in \NN ^n$ with $a_i\leq d_i$ we define a $K$-subspace of $V$ by 
$$W_a=\langle f_{ij} : i \in [n] \mbox{ and }  j  \in  [a_i] \rangle,$$
which clearly depends on the subspace  arrangement but also on the collection of bases  chosen. 
 
\begin{assumption} 
\label{assum}
Given  $\V=V_1,\dots, V_n$  we assume  that the collection of bases $\{ f_{ij} \}$ is  general in the sense that for all $a=(a_1,\dots, a_n)  \in \NN^n$ with $a_i\leq d_i$ the dimension of $W_a$ is the largest possible.  
\end{assumption} 

% The maximality of the  dimension of $W_a$ is a non-empty Zariski open condition on the irreducible  variety  that parametrizes the collections of bases of $\V$.  
 
A collection of bases satisfying \ref{assum} always exists (here we use the fact that the base field is infinite).  In other words, the subspace arrangement  can be special with inclusions and even equalities allowed,  but for each $V_i$ we pick a general basis.  

For later purposes we define two discrete polymatroids associated to the subspace arrangement 
$\V=V_1,\dots, V_n$.  For general facts and terminology on polymatroids we refer the reader to the classical paper by Edmonds \cite{E} and to monographs  \cite{F} and \cite{M} for  modern accounts. 
The subspace arrangement $\V$  gives rise to  the rank function  $\rk_\V: 2^{[n]}\to \NN$ defined by 
$$\rk_\V(A)=\dim_K \sum_{i\in A} V_i$$
 and the associated discrete polymatroid: 

$$\PM(\V)=\left\{ x\in \NN^n : \sum_{i\in A} x_i\leq \rk_\V(A) \mbox{ for all } A \subseteq [n]\right\}.$$

Let us set 
$$\PM(\V)^*=\left\{ x\in \NN^n : \sum_{i\in A} x_i\leq \rk_\V(A)-1 \mbox{ for all } \emptyset \neq A \subseteq [n]\right\}.$$

\begin{proposition} 
\label{DilTrun}
The set $\PM(\V)^*$  is a discrete polymatroid whose associated  rank  function is  the so-called Dilworth truncation   $\rk_\V^*: 2^{[n]}\to \NN$ of $\rk_\V$ defined as
$$\rk_\V^*(A)= \min\left\{ \sum_{i=1}^p \rk_\V(A_i)-p : A_1,\dots, A_p \mbox{  is a partition of } A\right\}$$
if $A\neq \emptyset$ and $\rk_\V^*(\emptyset)=0$.
 \end{proposition} 
  
In other words, $\rk_\V^*$ is the unique function satisfying properties (a),(b),(c),(d) in \cite[p.12]{E} such that: 

$$\PM(\V)^*=\left\{ x\in \NN^n : \sum_{i\in A} x_i\leq \rk_\V^*(A) \mbox{ for all } A \subseteq [n]\right\}.$$
In particular, 

$$\rk_\V^*(A)=\max\left\{ \sum_{i\in A} x_i :  x\in \PM(\V)^* \right\} $$
and 
$$\rk_\V^*([n])=\max\left\{ |x| : x\in \PM(\V)^*\right\}.$$

The assertion that $\PM(\V)^*$ is a polymatroid is a special case of Theorem 8 in Edmonds \cite{E}. 
A   proof of \ref{DilTrun} is obtained by combining Theorem 2.5 and Theorem 3.53  in  Fujishige's  monograph \cite{F}.  

We collect  now some simple facts about the vector spaces $W_a$ associated  to  a given subspace arrangement $\V$ and their relations with the two polymatroids just introduced. 

We have: 

\begin{lemma} \label{dependence}
Assume that there is a nontrivial linear dependence relation among the generators of $W_a$ involving one of the generators of $V_{q}$. Then $V_{q}\subseteq W_{a-e_{q}}$. 
\end{lemma}
\begin{proof}
For the given $q$ let $p$ be the largest index such that $f_{qp}$ appears in a nontrivial linear dependence relation among the generators of $W_a$.  This implies that $f_{qp}\in W_b$ with $b=(b_1,\dots, b_n)$ and 
$b_k=a_k$ for $k\neq q$ and $b_{q}=p-1$. But because of the choice of the $f_{ij}$'s  this implies that $V_{q}\subseteq W_b\subseteq W_{a-e_{q}}$. 
\end{proof}

\begin{lemma} \label{decomposition}
Set $T=\{ i\in [n]  : V_i \subseteq W_a \}$ and $b\in \NN^n$ with $b_i=0$ if $i\in T$ and $b_i=a_i$ otherwise. 
Furthermore set $c=a-b$. Then 
\begin{itemize}
\item[(1)] $W_a=W_b+  \sum_{i\in  T} V_i  $, 
\item[(2)] $\dim_K W_b=|b|$, i.e. the elements $f_{ij}$ with $i\not\in T$ and $j\leq a_i$ are linearly independent,  
\item[(3)] $W_b\cap  ( \sum_{i\in  T} V_i) =0$, 
\item[(4)] $W_c= \sum_{i\in  T} V_i$, 
\item[(5)] $\dim_K W_a=\sum_{i\not \in T} a_i+\rk_\V(T)$. 
\end{itemize} 
 
\end{lemma}
\begin{proof}
 (1) is obvious. (2) follows from Lemma  \ref{dependence} and the definition of $T$. 
 For (3) we set $u\in \NN^n$ with $u_i=d_i$ if $i\in T$ and $u_i=a_i$ otherwise. 
 Then we observe that, by (1) we have  $W_a=W_u$. If, by contradiction,  $W_b\cap  ( \sum_{i\in  T} V_i)$ is non-zero then there is a non-trivial linear relation among the generators of $W_u$ involving an element $f_{ij}$ with $i\not\in T$. Applying Lemma \ref{dependence} we have that $V_i\subseteq  W_u=W_a$, a contradiction with the definition of $T$. Finally (4) and (5) follow from   (1)-(3). 
 \end{proof} 

\begin{proposition} \label{formula}
We have: 
$$\dim_K W_a=\min \left \{   \sum_{i\not \in T} a_i+ \rk_\V(T)    :   T \subseteq [n] \right \}$$
\end{proposition}
\begin{proof} For every $T \subseteq [n]$ we have 
$$W_a\subseteq  \langle f_{ij} : i\not\in T \mbox{ and } j\leq a_i \rangle + \sum_{i\in T} V_i  $$
and therefore
$$\dim_K W_a \leq  \sum_{i\not\in T} a_i+\rk_\V(T).$$
It remains to prove that at least for one subset $T$ we have equality and this follows from Lemma \ref{decomposition} part (5).  
\end{proof} 

\begin{corollary} \label{corollary}
The following conditions are equivalent: 
\begin{itemize} 
\item[(1)] $\dim_K W_a=|a|$, i.e. the $f_{ij}$'s with $j\leq a_i$ are linearly independent.  
\item[(2)]  $ \sum_{i\in T} a_i \leq \rk_\V(T)$ for every $T\subseteq [n]$, i.e. $a\in \PM(\V)$.   
\end{itemize}   
\end{corollary}
\begin{proof} The implication $(1)\implies  (2)$ is obvious.  The implication $(2)\implies  (1)$   follows from Proposition \ref{formula}. 
\end{proof}

\begin{proposition} \label{notsub} The following conditions are equivalent: 
\begin{itemize} 
\item[(1)] for every $i$ one has  $V_i\not\subseteq W_a$. 
\item[(2)] for every  $\emptyset\neq T\subseteq [n]$ one has  $\sum_{i\in T} a_i \leq  \rk_\V(T)-1$, i.e. $a\in \PM(\V)^*$.
\end{itemize} 
\end{proposition}
\begin{proof}  $(1)\implies  (2)$: By virtue of Lemma  \ref{dependence} we know that the $f_{ij}$'s with $j\leq a_i$ are linearly independent. Hence for every non-empty $T\subseteq [n]$ we have 
$$\sum_{i\in T} a_i=\dim_K \langle f_{ij} : i\in T \mbox{ and } j\leq a_i\rangle \leq  \rk_\V(T)$$ 
and, if equality holds, we have $\sum_{i\in T} V_i\subseteq  W_a$ contradicting the assumption. 

$(2)\implies  (1)$. The assumption and Corollary \ref{corollary} imply that the $f_{ij}$'s's with $j\leq a_i$ are linearly independent. 
By contradiction suppose that $T=\{ i\in [n] : V_i\subseteq W_a\}$ is not empty.  
By Lemma \ref{decomposition} (5) we have
$$\dim_K W_a=\sum_{i\not\in  T} a_i+\rk_\V(T)$$ 
and by hypothesis $\rk_\V(T)>\sum_{i\in T} a_i$. It follows that 
$\dim_K W_a>|a|$ which is clearly a contradiction. 
\end{proof} 

%%%
\section{Ideals associated to subspace arrangements  and poset ideals} 
Given a subspace arrangement $\V=V_1,\dots, V_n$ of dimension $(d_1,\dots, d_n)$ we consider the ideal $I_i$ of $S$ generated by $V_i$ and set 
$$J=I_1I_2\cdots I_n.$$
We fix  a collection of bases  $f=\{ f_{ij} : i\in [n]  \mbox{ and }  j\in [d_i]\}$ of $\V$ satisfying  Assumption \ref{assum}.  
On $\NN^n$ we consider the standard poset structure defined as  $a \ge b$ if $a_i \ge b_i$ for every $i \in [n]$. Indeed $(\NN^n,\leq)$ is a distributive lattice with 
$$a\wedge b=(\min(a_1,b_1), \dots, \min(a_n,b_n))$$
 and
 $$a\vee b=(\max(a_1,b_1), \dots, \max(a_n,b_n)).$$ 
Consider the hypercube  $D=[d_1]\times \cdots \times [d_n]\subset \NN^n$ with the induced  poset structure. 
%and hence it is a graded poset with rank function $\rk(a)=|a|-n$.
%
 A poset ideal of $D$ is a subset $P\subseteq D$ such that if $a,b\in D$ and  $a\leq b\in P$ implies $a\in P$. 
 
For every $a\in D$ we set $f_a=\prod_{i=1}^n f_{ia_i}$ and observe that $J=( f_a : a \in D)$. Furthermore for $a\in \NN^n$ with  $a_i\leq d_i$ we   denote by $I_a$ the ideal of $S$ generated by the vector space $W_a=\langle f_{ij} : i \in [n] \mbox{ and } j \le a_i \rangle$. For every poset ideal $P$ of $D$ we define an ideal of the polynomial ring $S$ as follows: 
$$J_P=( f_a : a\in P).$$ 

Clearly $J_P$ depends on $\V$ but also on the collection of bases $f$ that we consider.  In particular $J=J_D$ and $J_\emptyset=\{0\}$.   
Let $a$ be a maximal element of a non-empty poset ideal $P$. Then $Q=P\setminus\{a\}$ is itself a poset ideal.  Furthermore set $b=a-(1,1,\dots,1)$. 
With this notation our first goal is to prove: 
\begin{theorem}
\label{main} \ 
\begin{itemize} 
\item[(1)] $J_P$ has a linear resolution. 
\item[(2)]  If  $f_a \not\in I_b$  then $J_Q:(f_a)= I_b$  and  if  $f_a \in  I_b$ then $f_a\in J_Q$ i.e.  $J_Q:(f_a) = S$. 
\end{itemize} 
 \end{theorem}
 
 \begin{proof} We prove the assertions by induction on the cardinality of $P$. Both assertions are obvious when $P$ has only one element. Note that (2) actually implies (1) because we have either $J_Q=J_P$ and we conclude by induction or we have  the short exact sequence 
 
$$
0\to S/I_b(-n)\to S/J_Q\to S/J_P\to 0 
$$
and again we conclude by induction.  
 So it remains to prove (2).  Set $A=\{ u\in D : u<a\}$. By construction $A\subseteq Q$ is a poset ideal and   
$$I_bf_a\subseteq J_A\subseteq J_Q \subseteq I_b.$$
 Hence 
 $$I_b\subseteq J_Q:f_a\subseteq I_b:f_a.$$
 Since $I_b$ is prime we have that $I_b=J_Q:f_a$ provided $f_a\not\in I_b$.

 It remains to prove that if $f_a\in I_b$ then actually $f_a\in J_Q$. Since $I_b$ is prime we have that $f_{ia_i}\in I_b$ for at least one $i\in [n]$ and this implies, by the choice of the $f_{ij}$'s, that  $V_i\subseteq   W_b$. Therefore the set $T=\{ i\in [n] : V_i\subseteq   W_b\}$ is not empty.  Up to a permutation of the coordinates we may assume that $T=\{1,\dots,m\}$ for some $m>0$. Set $a'=(a_1,\dots,a_m)$, $A'=\{ u'\in \NN^m : u'<a'\}$ and $b'=(b_1,\dots, b_m)$. 
 We have $I_{b'}\subseteq J_{A'}:f_{a'}$ by construction and $W_{b'}=\sum_{i\in [m]} V_i$ by Lemma \ref{decomposition} (4),   i.e. $I_{b'}$ is the maximal homogeneous ideal of the sub-polynomial ring $S'$ of $S$ generated by $\sum_{i\in [m]} V_i$. Since  the generators of $J_{A'}$ and $f_{a'}$ already belong to $S'$, we have that $f_{a'}$ is in the saturation of  $J_{A'}$ in $S'$. Note that $A'$ is a poset ideal of $D'=[d_1]\times \dots \times [d_m]$ and $|A'|\leq |A|<|P|$.  Hence, by induction, $J_{A'}$ has a linear resolution and therefore  it is saturated from degree $m$ and on. It follows that 
 $f_{a'}\in J_{A'}$ and then 
 $$f_a= f_{a'} \prod_{i=m+1}^n f_{ia_i}  \in J_{A'} \Bigg( \prod_{i=m+1}^n f_{ia_i} \Bigg) \subseteq J_A\subseteq J_Q$$
 as desired. 
\end{proof} 

Theorem \ref{main} has some important corollaries.  
We set 
$$D_\V=(1,\dots ,1)+\PM(\V)^*= \left\{ a\in D : \sum_{i\in T} a_i -|T|\leq \rk_\V(T)-1 \mbox{ for every } \emptyset \neq T\subseteq [n] \right\}.$$

\begin{corollary}\label{essential} Let $P$ be a poset ideal of $D$. Set $P'=P\cap D_\V$.  
We have  $J_P=J_{P'}$. In particular, $J=J_{D_\V}$. 
\end{corollary} 

\begin{proof} 
Using the notations of Theorem  \ref{main} we have seen that $f_a\in J_Q$ iff $f_a\in I_b$.  
The latter condition holds iff $V_i\subseteq  I_b$ for some $i$ and this is equivalent, in view of Proposition \ref{notsub}, to the the fact that $b\not\in \PM(\V)^*$.  In other words, if $a\in P\setminus D_\V$ then $f_a\in J_Q$, i.e. $J_P=J_Q$.  Iterating the argument one obtains  $J_P=J_{P'}$.
\end{proof} 

In view of Corollary \ref{essential}  when studying the ideal $J_P$ we may assume that $P\subseteq D_\V$.

\begin{corollary}\label{linq} Let $P\subseteq D_\V$ be a poset ideal.    
We have: 
\begin{itemize}
 \item[(1)] $J_P$ has linear quotients. More precisely, any total order on $P$ that refines the partial order $\leq$ gives rise to a total order on the generators of $J_P$ that have linear quotients.  
\item[(2)]  We have: 
$$\sum_{j\geq 0} \beta_i (J_P) z^i=\sum_{a\in P} (1+z)^{|a|-n}
.$$
\end{itemize} 
\end{corollary} 

\begin{proof} (1)  follows immediately from Theorem  \ref{main} part  (2) while (2) follows from the short exact sequence used in the proof of Theorem  \ref{main}. 
\end{proof}

  Let us single out the special case  

\begin{corollary}\label{linq1} 
\begin{itemize} 
\item[(1)] The ideal $J$ is minimally generated by $f_a$ with $a\in D_\V$.
\item[(2)] The ideal $J$ has linear quotients. Indeed ordering the generators $f_a$ with $a\in D_\V$   according to a linear extension of the partial order $\leq$ gives linear quotients. 
\item[(3)] The Betti numbers of $J$ are given by the formula:   
$$\sum_{i\geq 0} \beta_i (J) z^i=\sum_{a\in D_\V} (1+z)^{|a|-n}=\sum_{j\geq 0} \gamma_i(\V)(1+z)^{i}$$
where $\gamma_i(\V)=  \# \{x\in \PM(\V)^* : |x|=i\}  $. 
\item[(4)] The projective dimension $\pd J$ of $J$ is $\rk^*_\V([n])$, i.e. 
$$\pd J=\min\left \{  \sum_{i=1}^p   \rk_\V(A_i)-p : A_1,\dots, A_p \mbox{ is a partition of } [n] \right \}.$$
\end{itemize} 
\end{corollary} 

%%%%%%%%%%
\section{Irredundant primary decomposition and stability of associated primes} 
We keep the notation of the previous section. As noted earlier, the key ingredient in proving $\reg(J)=n$ was a description given in \cite{CH} of a (possibly  redundant) primary decomposition of $J$, i.e. 
$$ J = \bigcap_{\emptyset \neq A \subseteq [n]} I_A ^{\# A}$$ 
where for $A \subseteq [n]$ we have set  $I_A = \sum_{i \in A} I_i$. 
Here one notes that $I_A$ is an ideal generated by linear forms and hence prime with primary powers. 
The first reason why the decomposition can  be redundant is that different components might have the same radical. We consider the set of the so-called flats of the polymatroid $P(\V)$, i.e. 
$$F(\V)=\{  B\subseteq [n] :  \rk_\V(B)<\rk_\V(A) \mbox{ for all } 
B\subsetneq A \subseteq [n]  \}$$ and observe that  if $A\subseteq  [n]$  and $B$ is its closure, i.e. 
$B=\{ i : \rk_\V(A)=\rk_\V(A\cup \{i\} )\}\in F(\V)$  then $I_A ^{\# A}\supseteq   I_B ^{\# B}$. Hence 
$$J = \bigcap_{B \in F(\V)} I_B^{\# B}$$
is still a primary decomposition and now the  radicals of the components are  distinct. 
To get an irredundant primary decomposition it is now enough to identify   for which $B \in F(\V)$ the prime  ideal  $I_B$ is  associated to $J$. 

\begin{proposition}
\label{asso1} 
For $B\in F(\V)$ we have that the prime ideal $I_{B}$ is associated to $J$  if and only if $\rk_{\V}^*(B) = \rk_{\V}(B) - 1$. 
\end{proposition}
\begin{proof}
 Set $P=I_B$.  Since $B\in F(\V)$, we have that  $JS_P = \prod_{i \in B} I_i S_P$. We have that $P$ is associated to $J$ if and only if  $P$  is associated to  $\prod_{i \in B} I_i $.  So we may assume right away that $B=[n]$ and $I_{[n]}$ is the graded maximal ideal of $S$.  
By part (4) of Corollary \ref{linq1} we have $\pd J= \rk^*_{\V}([n])$ and by the Auslander-Buchsbaum formula 
$\pd J=\rk_{\V}([n])-1$ if and only if  $I_{[n]} \in \Ass(S/J)$. Hence $I_{[n]} \in \Ass(S/J)$ if and only if $\rk^*_{\V}([n]) = \rk_{\V}([n])-1$. 
\end{proof} 

Summing up we have: 

\begin{theorem}
\label{primdecasso}  An   irredundant primary decomposition of $J$ is given by 
$$ J = \bigcap_{B} I_{B}^{\# B}$$ 
where $B$ varies in the set $\{ B\in F(\V) : \rk_{\V}^*(B) = \rk_{\V}(B)-1\}$. In particular, 
$$\Ass(S/J)=\{ I_B :  B\in F(\V) \mbox{ and } \rk_{\V}^*(B) = \rk_{\V}(B)-1   \}.$$
\end{theorem}
\begin{proof}
To obtain an irredundant primary decomposition of $J$ it is enough to remove from the possibly redundant  primary decomposition $J = \bigcap_{B \in F(\V)} I_B^{\# B}$ the components not corresponding to associated primes. Hence by \ref{asso1} we get the  irredundant primary decomposition described in the statement. The assertion about the associated primes  in then an immediate consequence. 
\end{proof} 

\begin{corollary}
Suppose that the  subspace arrangement $\V=V_1,\dots, V_n$ of $V$   is linearly general, i.e. 
$\dim \sum_{i\in A} V_i  =\min\{  \sum_{i\in A} d_i,  d\}$  for all $A\subseteq  [n]$ where $d_i=\dim V_i$ and $d=\dim V$.  Assume $n>1$ and $d_i<d$ for all $i$  and set $I_i=(V_i)$.  We have  $\Pi_{i=1}^n I_i=\cap_{i=1}^n I_i$ if and only if $d_1+d_2+\dots+d_n < d+n-1$. 
\end{corollary} 
\begin{proof}  If $d_1+d_2+\dots+d_n\leq d$ then the assertion is obvious. So we may assume $d_1+d_2+\dots+d_n> d$. In particular, $I_{[n]}$ is the maximal ideal $\m$ of $S$ and $\rk_{\V}([n])=d$.   It has been already observed in \cite{CH} that for a linearly general subspace arrangement a  primary decomposition of the product ideal $J$ is given by $J=\cap_{i=1}^n I_i \cap \m^n$.  Therefore we have that $J=\cap_{i=1}^n I_i$ if and only if $\m$ is not associated to $J$. In view of the characterization given in \ref{primdecasso}, the latter is equivalent to  $\rk_{\V}^*([n]) < \rk_{\V}([n])-1=d-1$, that is, 
$\sum_{i=1}^p \rk_{\V}(A_i)-p< d-1$ for some  partition $A_1,\dots, A_p$ of $[n]$. 
Summing up, we have to prove that the  following conditions are equivalent: 
\begin{itemize} 
\item[(1)] $d_1+d_2+\dots+d_n < d+n-1$ 
\item[(2)]  $\sum_{i=1}^p \rk_{\V}(A_i)-p< d-1$ for some  partition $A_1,\dots, A_p$ of $[n]$. 
\end{itemize} 
That (1) implies (2) is clear, just take $p=n$ and  $A_i=\{i\}$. Vice versa, let $A_1,\dots, A_p$ be a partition  of $[n]$ such that $\sum_{i=1}^p \rk_{\V}(A_i)-p< d-1$. If $\sum_{j\in A_v} d_j\geq d$ for some $v$ one has $\rk_{\V}(A_v)=d$, contradicting the assumption.  Hence $\sum_{j\in A_i} d_j< d$ for all $i$. By assumption this implies that $\rk_{\V}(A_i)=\sum_{j\in A_i} d_j$ for all $i$. It follows that $d_1+d_2+\dots+d_n=\sum_{i=1}^p \rk_{\V}(A_i)<d-1+p\leq d-1+n$ as desired. 
\end{proof}

Now we turn our attention to the properties of the powers $J^\nuu$ of the ideal $J$ with $\nuu>0$.  Clearly $J^\nuu$ is associated to  the subspace arrangement $\V^\nuu=\{ V_{ij} :   (i,j) \in [n]\times [\nuu]\}$ with $V_{ij}=V_i$ for all $j$.   The  polymatroids and rank functions associated to $\V^\nuu$ are very tightly related to those of $\V$ as we now explain. Since $\V^\nuu$ is indexed  on $[n]\times [\nuu]$  the domain of   the  associated rank function  $\rk_{\V^\nuu}$ is $2^{[n]\times [\nuu]}$.  Let $\pi: [n]\times [\nuu] \to [n]$  be the projection on the first coordinate. We have: 

\begin{lemma}
\label{simple} 
For every subset  $A\subseteq [n]\times [\nuu]$ we have 
 $$\rk_{\V^\nuu}(A)=\rk_{\V}(\pi (A))=\rk_{\V^\nuu}(\pi^{-1}\pi (A))$$
 and 
 $$\rk^*_{\V^\nuu}(A)=\rk^*_{\V}(\pi (A))=\rk^*_{\V^\nuu}(\pi^{-1}\pi (A)).$$
 \end{lemma} 
 \begin{proof} For the first assertion one observes that 
 $$\rk_{\V^\nuu}(A)=\dim \sum_{(i,j)\in A} V_{i,j}=\dim \sum_{i\in \pi(A) } V_{i}=\rk_{\V}(\pi(A)).$$ 
 For the second, set $\nu=\rk^*_{\V^\nuu}(A)$ and let $A_1,\dots, A_p$ be a partition of $A$ such that 
 $\nu=\sum_{c=1}^p \rk_{\V^\nuu}(A_c)-p$. If for some $i,j$ one has $\pi(A_i)\cap \pi(A_j)\neq \emptyset$ then let $k\in \pi(A_i)\cap \pi(A_j)$. Let $A_1',\dots, A_q'$ be obtained from $A_1,\dots, A_p$ by replacing  $A_i$ with $A_i\cup \{ (a,b)\in A_j : a=k\}$ and $A_j$ with $ \{ (a,b)\in A_j : a\neq k\}$ if  $\{ (a,b)\in A_j : a\neq k\}\neq \emptyset $ (in this case $q=p$), or simply by removing $A_j$ if $ \{ (a,b)\in A_j : a\neq k\}=\emptyset $ (and in this case $q=p-1$). One can check that the new partition satisfies $\sum_{c=1}^q \rk_{\V^\nuu}(A'_c)-q\leq \nu$ and hence  $\sum_{c=1}^q \rk_{\V^\nuu}(A'_c)-q=\nu$. We may repeat the process until we obtain a partition $A_1,\dots, A_s$  of $A$ such that   $\nu=\sum_{c=1}^s \rk_{\V^\nuu}(A_c)-s$  and $\pi(A_i)\cap \pi(A_j)= \emptyset$ for every $i\neq j$. Then $\pi(A_1),\dots, \pi(A_s)$ is a partition of $\pi(A)$ and 
 $\rk^*_{\V^\nuu}(A)=\nu=\sum_{c=1}^s \rk_{\V^\nuu}(A_c)-s=\sum_{c=1}^s \rk_{\V}(\pi(A_c))-s\geq \rk^*_\V(\pi(A))$.  
 Vice versa if $B_1,\dots, B_s$ is a partition on $\pi(A)$ such that $\sum_{c=1}^s \rk_{\V}(B_c)-s=\rk^*_\V(\pi(A))$ then with $A_i=A\cap \pi^{-1}(B_i)$ one gets a partition $A_1,\dots, A_s$ of $A$ such that 
 $\rk^*_{\V^\nuu}(A) \leq \sum_{c=1}^s \rk_{\V^\nuu}(A_c)-s=\rk^*_\V(\pi(A))$. 
\end{proof} 

We obtain: 

\begin{theorem}  
\label{powersofJ}
For every $\nuu>0$ we have: 
\begin{itemize}
\item[(a)] $\pd J = \pd J^\nuu$, 
\item[(b)]$\Ass(S/J)=\Ass(S/J^\nuu)$,
\item[(c)] an irredundant primary decomposition of  $J^\nuu$ is obtained by raising to power $\nuu$ the components in the irredundant primary decomposition of  $J$ described in \ref{primdecasso}, i.e. 
$$J^\nuu  = \bigcap_{B} I_{B}^{\nuu \# B}$$
where $B\in F(\V)$ and  $\rk_{\V}^*(B) = \rk_{\V}(B)-1$.
\end{itemize} 
\end{theorem}
\begin{proof} (a) By \ref{linq1}(4)   $\pd J=\rk^*_\V([n])$ and $\pd J^\nuu=\rk^*_{\V^\nuu}([n]\times [\nuu])$ and by \ref{simple}  $\rk^*_\V([n])=\rk^*_{\V^\nuu}([n]\times [\nuu])$. 

Assertions (b) and (c):  by \ref{primdecasso}  the associated primes of  $J^\nuu$ arise form subsets $C\subseteq [n]\times [\nuu]$ such that 
$\rk^*_{\V^\nuu}(C)=\rk_{\V^\nuu}(C)-1$ and $C\in F(\V^\nuu)$, i.e.~$\rk_{\V^\nuu}(C)<\rk_{\V^\nuu}(A)$ for all $C\subsetneq A$. 
The second condition together with \ref{simple} implies that  $C=\pi^{-1}(B)$ with $B=\pi(C)$. But then, again by \ref{simple},  $\rk^*_{\V^\nuu}(C)=\rk_{\V^\nuu}(C)-1$  is equivalent to $\rk^*_{\V}(B)=\rk_{\V}(B)-1$. Summing up, $F(\V^\nuu)=\{ \pi^{-1}(B) : B \in F(\V)\}$ and hence the associated primes of $J^\nuu$ are exactly the associated primes of $J$.  The assertion concerning the primary decomposition follows immediately since  $ \# \pi^{-1}(B)=\nuu \# B$. 
\end{proof} 

The established relations \ref{simple} among the rank functions translate immediately to the following relation  involving  the associated polymatroids: 

\begin{proposition}
\label{polymatv}
For every $\nuu$ we have: 
$$P(\V^\nuu)^*=\left\{ (x_{ij} ) \in \NN^{[n]\times [\nuu]} :  \left(\sum_{j \in [\nuu]} x_{1j}, \dots, \sum_{j \in [\nuu]} x_{nj}\right) \in P(\V)^*\right\}.$$
\end{proposition} 

Since the Betti numbers can be expressed in terms of the points in $P(\V^\nuu)^*$,  using \ref{polymatv}  one can deduce a formula for the Betti numbers of $J^\nuu$ that just depends on $P(\V)^*$: 

\begin{corollary} For every $\nuu >0$ and every $i\geq 0$ one has: 
$$\beta_i(J^\nuu) =  \sum_{x \in P(\V)^*}  {|x|  \choose i}   \prod_{j=1}^n   {\nuu+x_j-1 \choose x_j} $$ 
\end{corollary} 

\begin{remark} 
\label{propow} 
As a further generalization, instead of the powers $J^\nuu$ of $J=I_1I_2\cdots I_n$ one can consider a product of powers of the $I_i$'s, i.e,   $I_1^{u_1}\cdots I_n^{u_n}$ with $(u_1,\dots, u_n)\in \NN^n$ and the arguments we have presented extend immediately.  Assuming  $u_i>0$ for all $i$ one has: 
\begin{itemize} 
\item[(a)]    the results in  \ref{powersofJ} (a), (b), (c)  hold with the ideal $J^\nuu$ replaced by $I_1^{u_1}\cdots I_n^{u_n}$ and the exponent $\nuu \# B$ replaced by $\sum_{i\in B} u_i$.  
\item[(b)]  The polymatroid  associated to the subspace arrangement $\V^{(u_1,\dots, u_n)}= \{ V_{ij} \}$ with $V_{ij}=V_i$ for all $j\in [u_i]$ is: 
$$P(\V^{(u_1,\dots, u_n)})^*=\left\{ (x_{ij} ) \in \NN^{[u_1]\times\dots  \times [u_n]}  :  \left(\sum_{j \in [u_1]} x_{1j}, \dots, \sum_{j \in [u_n]} x_{nj}\right) \in P(\V)^*\right\}.$$
\item[(c)] The formula for the Betti numbers  is: 
$$\beta_i(I_1^{u_1}\cdots I_n^{u_n}) =  \sum_{x \in P(\V)^*}  {|x|  \choose i}   \prod_{j=1}^n   {u_j+x_j-1 \choose x_j}$$ 
\end{itemize} 
\end{remark} 

The case  $i=0$ of the  formula \ref{propow}(c)  deserves a special attention because  of its relation with the so-called multiview variety that arises in geometric computer vision. 
Let us recall from  \cite{AST, C, CS, CDG, Li}  that the subspace arrangement $\V$ defines a multiprojective variety whose coordinate ring can be identified with  the subring 
$$A=K[V_1y_1,\dots, V_ny_n]$$ of the Segre product  $K[x_iy_j : i=1,\dots,d   \mbox{ and }  j=1,\dots n]$. The ring $A$ is  $\ZZ^n$-graded by $\deg y_j=e_j\in \ZZ^n$. 
Given  $u=(u_1,\dots, u_n)\in \NN^n$, the $u$-th  homogeneous  component $A_u$ of $A$ is  $V_1^{u_1}\cdots V_n^{u_n}$ and its dimension equals to $\beta_0(I_1^{u_1}\cdots I_n^{u_n})$.  So we get a relatively simple   and new proof of an improved  version of the main result of \cite{Li}: 

\begin{theorem}
For every $u=(u_1,\dots, u_n)\in \NN^n$ the multigraded Hilbert function of the coordinate ring $A=K[V_1y_1,\dots, V_ny_n]$ of the multiview variety associated with the subspace arrangement $\V=\{V_1,\dots, V_n\}$ is given by: 
$$\dim_K A_{u}=\sum_{x \in P(\V)^*}    \prod_{j=1}^n   {u_j+x_j-1 \choose x_j} $$ 
In particular, the multidegree of $A$ is multiplicity free and supported on the maximal elements of the polymatroid $P(\V)^*$. 
%and the multigraded Hilbert series $\HS(A,z_1,\dots,z_n)$ of  is given by:
%$$\HS(A,z_1,\dots,z_n)=\sum_{x \in P(\V)^*}  \frac{ \prod_{x_i>0} z_i}{  \prod_{i=1}^n (1-z_i)^{1+x_i} }$$
\end{theorem}

\section{Resolution of $J_P$} 
For every subspace arrangement $V_1,\dots, V_n$ of dimension $(d_1,\dots, d_n)$ with a given collection of bases $f=\{ f_{ij} : i\in [n]  \mbox{ and } j\leq d_i\}$ satisfying Assumption \ref{assum} and for every poset ideal $P$ of $D=[d_1]\times \cdots \times [d_n]$ we have proved that the ideal $J_P$ has a linear resolution and that the Betti numbers are combinatorial invariants. Our goal is now to describe explicitly  a minimal free  resolution of $J_P$. We start with the ``generic" case. 

\subsection{Resolution of $J_P$:  the generic case} 
Assume firstly that, for the given $(d_1,\dots, d_n)$, the $V_i$'s  are as generic as possible. 
That is, we assume that there is a basis $\{ x_{ij} : i\in [n] \mbox{ and } j\in [d_i] \}$ of the ambient vector space such that $V_i$ is generated by $\{ x_{ij} :   j\in [d_i] \}$. Note that the collection of bases $
x=\{ x_{ij} : i\in [n] \mbox{ and } j\in [d_i] \}$  satisfy the  Assumption \ref{assum} and we will consider the ideals $J_P$ with respect to $x$.  In this case 
$$S=K[x_{ij} :  i\in [n] \mbox{ and } j\in [d_i] ].$$ The corresponding ideal $J$ is the product of  transversal ideals $I_i=( x_{ij}:  \, j\in [d_i]  )$ because each factor  uses a different set of variables.  Then the resolution of $J$ is given by the tensor product of the resolutions of the $I_i$'s, i.e. the (truncated) Koszul complex on the set $x_{ij}$ with $j\in [d_i]$. More explicitly, let $\K^{(i)}$ be the Koszul complex on $x_{ij}$ with $j\in [d_i]$ with the $0$-th component removed and homologically shifted so that 
$$\K^{(i)}_j=\wedge^{j+1} S^{d_i}.$$ 
This is sometimes called the first syzygy complex of the full Koszul complex. Denote by $e_{i1},\dots, e_{id_i}$ the canonical basis of  $S^{d_i}$. For every non-empty subset $A_i=\{j_1, j_2,\dots \}$ of $[d_i]$  with $j_1<j_2<\dots$ we have the corresponding basis element 
$e_{A_i}=e_{ij_1}\wedge e_{ij_2} \wedge \dots $ of $\K^{(i)}$ in homological degree $|A_i|-1$. Then 
$$\K=\K^{(d_1,\dots,d_n)}=\K^{(1)}\otimes \K^{(2)}\otimes \cdots \otimes \K^{(n)}$$
is the free resolution of $J=I_1\cdots I_n$. An $S$-basis of $\K$ can be described as follows. Let $A=(A_1,\dots, A_n)$ with $A_i$ a non-empty subset of $[d_i]$. Set 
$e_A=e_{A_1}\otimes   e_{A_2}\otimes \cdots \otimes   e_{A_n}\in \K$. Then the homological degree of $e_A$ is $\sum_{i=1}^n |A_i|-n$ and the set of all $e_A$'s form an $S$-basis of $\K$. The differential  $\partial_{\K}$  of $\K$ can be described  as follows: 
$$\partial_{\K}(e_A)=\sum_{i\in [n], |A_i|>1} \ \sum_{b\in A_i}  (-1)^{\sigma(i,b)} x_{ib} \ e_{A_1}\otimes \dots  \otimes  e_{A_i\setminus \{b\}} \otimes \cdots \otimes   e_{A_n}$$ 
where 
$$\sigma(i,b)=\sum_{j<i} (|A_j|-1)+|\{ c\in A_i : c<b\}|.$$

  For a given poset ideal $P$ of $D$ we define 

$$\K_P=\K^{(d_1,\dots,d_n)}_P=\oplus S e_A$$
where the sum is extended to all the $e_A$ such that $(\max(A_1), \dots, \max(A_n))\in P$.  
Clearly $\K_P$ is a subcomplex of $\K$ and $(\K_P)_0=\oplus_{a\in P}  S e_{1a_1} \otimes e_{2a_2} \cdots  \otimes e_{na_n}$ and our goal is to prove: 

\begin{theorem}
\label{resgen} 
The complex $\K_P$ is a minimal free resolution of $J_P$. 
\end{theorem} 

  Augmenting the complex $\K_P$ with the map 
$$(\K_P)_0\to S$$ 
sending $e_{1a_1} \otimes e_{2a_2} \cdots  \otimes e_{na_n}$ to $f_a=x_{1a_1}\dots x_{na_n}$ one gets a complex $\tilde \K_P$ and we will actually prove it is a resolution of $S/J_P$. We need the following properties that follow immediately from the definitions.    

\begin{remark} 
\label{remobv}\
\begin{itemize}
\item[(1)]  An inclusion $P_1\subseteq  P_2$ of poset ideals  of $D$  induces  an inclusion of the associated  complexes $\tilde \K_{P_1} \subseteq \tilde \K_{P_2}$.  
\item[(2)] Given two poset ideals $Q_1,Q_2$ of $D$   both $Q_1\cup Q_2$ and $Q_1\cap Q_2$ are poset ideals and  one has 
 $\tilde \K_{Q_1} \cap \tilde \K_{Q_2}=\tilde \K_{Q_1\cap Q_2}$ and  $\tilde \K_{Q_1} + \tilde \K_{Q_2}=\tilde \K_{Q_1\cup Q_2}$. 
 \item[(3)] Given two poset ideals $Q_1,Q_2$ of $D$ one has a short exact sequence of complexes 
 $$0\to \tilde \K_{Q_1\cap Q_2}  \to \tilde \K_{Q_1} \oplus \tilde \K_{Q_2} \to  \tilde \K_{Q_1\cup Q_2} \to 0$$
where the first map sends $y$ to $(y,y)$ and the second sends $(y,z)$ to $y-z$. 
\end{itemize} 
 \end{remark}

 Later on we will also need the following assertion that is part of the folklore of the subject. 
 
\begin{lemma}
\label{regseq} 
Let $S$ be a positively graded ring and $M$ a finitely generated graded $S$-module. 
Let $x_1,\dots, x_h$ be elements of degree $1$ of $S$ and set $I=(x_1,\dots,x_h)$. Denote by $\HS(M,z)$ the Hilbert series of $M$.  Assume $\HS(M/IM,z)=\HS(M,z)(1-z)^h$. Then  $x_1,\dots, x_h$ is an $M$-regular sequence. 
\end{lemma}
 
\begin{proof}
For $i=0,1,\dots, h$ we  set $I_i=(x_1,\dots,x_i)$ and $N_i=M/I_iM$. Denote by $T_i$ the kernel of multiplication by 
$x_{i+1}$  on $N_i$. For $i<h$ we have an exact sequence: 
$$0\to T_i\to N_i(-1)\to N_i\to N_{i+1}\to 0$$
and hence 
$$\HS(N_{i+1},z)=\HS(N_{i},z)(1-z)+\HS(T_{i},z)$$
Taking into consideration that $N_0=M$ it follows that for every $j\geq 0$ one has
$$\HS(N_{j},z)=\HS(M,z)(1-z)^j+ \sum_{i<j}\HS(T_{i},z)(1-z)^{j-1-i}.$$
Setting $j=h$ and using the assumption one has: 
$$\sum_{i<h}\HS(T_{i},z)(1-z)^{h-1-i}=0$$
Since $\HS(T_{i},z)$ are series with non-negative terms and the least degree component of $(1-z)^{h-1-i}$ is positive, $
\HS(T_{i},z)=0$ for every $i$, that is $T_i=0$ for every $i$. 
\end{proof}

\begin{theorem}
\label{resgen1} 
The complex $\tilde \K_P$ is a minimal free resolution of $S/J_P$. 
\end{theorem}

\begin{proof} By construction we have that $H_0(\tilde \K_P)=S/J_P$ and hence we have to show that $H_i(\tilde \K_P)=0$ for $i>0$. We do it by induction on $|P|$. The case $|P|=1$ is obvious. 
Let $M$ be the set of maximal elements in $P$. 

If $|M|=1$, say $M=\{a\}$  with $a=(a_1,\dots, a_n)$, then $P=\{ b\in D : b\leq a\}$ and $J_P=\prod_{i=1}^n  (x_{i1},\dots, x_{ia_i})$. Then a resolution of $S/J_P$ is given by the augmented complex obtained by the tensor product of the truncated Koszul complexes associated to $x_{i1},\dots, x_{ia_i}$ which is exactly $\tilde \K_P$. 

If instead $|M|>1$, say $M=\{m_1,\dots, m_v\}$ set $Q_1=\{ b\in D : b\leq m_i \mbox{ for some } i<v\}$ and $Q_2=\{ b\in D : b\leq m_v\}$ so that $P=Q_1\cup Q_2$.  By \ref{remobv}(3) we have  a short exact   sequence of complexes:  

$$0\to \tilde \K_{Q_1\cap Q_2}  \to \tilde \K_{Q_1} \oplus \tilde \K_{Q_2} \to  \tilde \K_{P} \to 0.$$

  The associated long exact sequence on homology together with the fact that, by induction, we already know the statement for $Q_1, Q_2$ and $Q_1\cap Q_2$, imply that $H_i(\tilde \K_{P})=0$ for $i>1$ and that $H_1(\tilde \K_{P})$ fits in the exact sequence:

$$0\to H_1(\tilde \K_{P}) \to S/J_{Q_1\cap Q_2} \to  S/J_{Q_1} \oplus  S/J_{Q_2} \to S/J_P\to 0$$ 

  But  $J_{Q_1\cap Q_2}=J_{Q_1}  \cap J_{Q_2}$ and $J_P=J_{Q_1}  + J_{Q_2}$ because of Lemma  \ref{interlemma} and then it follows that $H_1(\tilde \K_{P})$   vanishes as well. 
\end{proof} 

 \begin{lemma} 
 \label{interlemma}
Let $P_1,P_2$ be poset ideals of $D$. Then $J_{P_1\cap P_2}=J_{P_1}  \cap J_{P_2}$  and 
$J_{P_1\cup P_2}=J_{P_1} + J_{P_2}$.
\end{lemma} 

\begin{proof} 
The second assertion and   the inclusion $J_{P_1\cap P_2}\subseteq J_{P_1}  \cap J_{P_2}$ are obvious. For the other inclusion, since the ideals involved are monomial ideals, the intersection $J_{P_1}  \cap J_{P_2}$ is generated by $\LCM(f_a,f_b)$ with $a\in P_1$ and $b\in P_2$. But $f_{a\wedge b} | \LCM(f_a,f_b)$ and $a\wedge b\in P_1\cap P_2$. 
\end{proof} 

\subsection{Resolution of $J_P$: arbitrary configurations} 

Now let us return to the case of an arbitrary  subspace arrangement $\V=V_1,\dots, V_n$ of dimension $(d_1,\dots, d_n)$ and fix a collection of bases $\{ f_{ij} \}$  satisfying Assumption \ref{assum}.  
Consider the $K$-algebra map: 

$$T=K[x_{ij} : i\in [n] \mbox{ and } j\in [d_i] ]\to S$$
sending $x_{ij}$ to $f_{ij}$ which, without loss of generality, we may assume is surjective. We consider $S$ as a $T$-module via this map.  We have: 

\begin{theorem}
\label{main2}
For every poset ideal $P\subseteq  D_\V$ the complex $\tilde \K_P \otimes_T S$ is a minimal $S$-free resolution of $S/J_P$. 
\end{theorem} 

\begin{proof} In the proof we need to distinguish the ideal $J_P$ associated with the arbitrary subspace arrangement $V_1,\dots, V_n$  and collection of bases $f$ with the one, that we will denote by  $J_P^g$, associated with the generic arrangement of dimension $(d_1,\dots, d_n)$ and collection of bases  $x$. Let $U$ be the kernel of the map $T\to S$.
 By construction, $U$ is generated by $h=\sum_{i=1}^n d_i-\dim_K \sum_{i=1}^n V_i$ linear forms and 
  one has 
  $$T/J_P^g\otimes_T  S=T/(J_P^g+U)=S/J_P.$$
Since by Theorem \ref{resgen1} $\tilde \K_P$ is a resolution of  $T/J_P^g$  it is enough to prove that the generators of $U$ form a $T/J_P^g$-regular sequence. 
Note that by Corollary \ref{linq} $T/J_P^g$ and $S/J_P$ have the same Betti numbers and hence their Hilbert series differ only by the factor $(1-z)^h$. Then by Lemma \ref{regseq} one concludes that the generators of $U$ form a $T/J_P^g$-regular sequence. 
\end{proof}

As a consequence we have that Lemma \ref{interlemma} holds for arbitrary subspace configurations: 

\begin{corollary} \label{interlemma2}
Let $P_1,P_2$ be poset ideals of $D_\V$. Then $J_{P_1\cap P_2}=J_{P_1}  \cap J_{P_2}$  and 
$J_{P_1\cup P_2}=J_{P_1} + J_{P_2}$.
\end{corollary} 
\begin{proof} 
The second assertion and the inclusion $J_{P_1\cap P_2}\subseteq J_{P_1}  \cap J_{P_2}$ are obvious. The short exact sequence of complexes 
$$0\to \tilde \K_{P_1\cap P_2}\otimes S   \to  ( \tilde \K_{P_1} \otimes S ) \oplus ( \tilde \K_{P_2} \otimes S ) \to  \tilde \K_{P_1\cup P_2} \otimes S \to 0$$ 
induces an exact sequence in homology that, by virtue of Theorem \ref{main2}, yields the following short exact sequence: 
$$0\to S/J_{P_1\cap P_2} \to S/J_{P_1} \oplus S/J_{P_2} \to  S/J_{P_1}+J_{P_2} \to 0$$
that in turns implies the desired equality. 
\end{proof}

As a special case of Theorem \ref{main2} we have: 

\begin{theorem}
\label{main3}
For every subspace arrangement $\V=V_1,\dots, V_n$  the complex $\tilde \K_{D_\V}  \otimes_T S$ is a minimal $S$-free resolution of $S/J$. 
\end{theorem} 

\begin{remark}
The formulas for the Betti numbers and projective dimension hold over any base field. 
The resolution described works provided  the base field is infinite. 
\end{remark} 

\medskip

%%%%%%%%%%%%%%%%%%%%%%%%%%%%%%%%%%%%%%%%%%%%%%%%%%%%%%%%%%%%%%%%%%%%%%%%%%

\end{document}